\documentclass{amsart}
\usepackage{amscd,amsfonts,amsmath,amssymb,amsthm}
\usepackage{xcolor, graphicx, ytableau} 
\usepackage[colorlinks]{hyperref}
\usepackage[normalem]{ulem}
\usepackage{tikz}
\usepackage{chessfss}
\usetikzlibrary{calc, fit, backgrounds}
\newlength{\cellsize}
\definecolor{yellow}{rgb}{0.5,0.5,0}
\definecolor{green}{rgb}{0,0.5,0}
\definecolor{orange}{rgb}{1, .57, 0.09}
\definecolor{color3}{rgb}{0.53125,0.796875
,0.9296875}
\definecolor{color2}{rgb}{0.86328125,0.796875,0.46484375}
\definecolor{color1}{rgb}{0.796875,0.3984375,0.46484375}
\definecolor{lc3}{rgb}{0.8828125,0.94921875
,0.982421875}
\definecolor{lc2}{rgb}{0.9658203125,0.94921875,0.8662109375}
\definecolor{lc1}{rgb}{0.94921875,0.849609375,0.8662109375}

\newcommand{\blacksquarenew}{\raisebox{-.86pt}{$\blacksquare$}}

\usepackage{skak}  
\usepackage{graphicx}

\usepackage{multicol}
\usepackage{geometry}

\usepackage{chessfss}
\newcommand{\Rook}{\raisebox{-1.5pt}{\scalebox{.4}{\BlackRookOnWhite}}}
\newcommand{\WRook}{\raisebox{-1.5pt}{\scalebox{.4}{\WhiteRookOnWhite}}}

\theoremstyle{plain}
\newtheorem{conj}{Conjecture}[section]
\newtheorem{prop}[conj]{Proposition}
\newtheorem{proposition}[conj]{Proposition}

\newtheorem{theorem}[conj]{Theorem}
\newtheorem{lemma}[conj]{Lemma}
\newtheorem{corollary}[conj]{Corollary}

\theoremstyle{remark}
\newtheorem{example}[conj]{Example}
\newtheorem{remark}[conj]{Remark}
\newtheorem{question}[conj]{Question}

\numberwithin{equation}{section}

\newcommand{\op}{\operatorname}
\newcommand{\LL}{\textsf{L}}
\newcommand{\Graph}[1]{\op{Graph}(#1)}
\newcommand{\Grid}[1]{\op{Grid}(#1)}
\newcommand{\RP}{\op{RP}}
\newcommand{\perm}{\op{perm}}
\newcommand{\len}{\op{len}}

\definecolor{darkgreen}{RGB}{55,138,0}

\title{Non-attacking rook placements on crossword grids}
\author{Joel Brewster Lewis and Robert Won}
\address{Department of Mathematics, The George Washington University, Washington, DC 20052, USA}
\email{jblewis@gwu.edu}
\email{robertwon@gwu.edu}
\date{}

\begin{document}

\begin{abstract}
We introduce the notion of a non-attacking rook placement on a crossword grid. A crossword grid is a collection of white squares (which comprise across and down words) and black squares (which separate the words), and a complete non-attacking rook placement on such a grid is a subset of white squares which intersects every across and every down word exactly once. We prove an upper bound on the number of rook placements that a general grid can admit. We then study sparse grids in which no two black squares share an edge and show that rook placements on certain sparse grids correspond bijectively to alternating sign matrices with prescribed $-1$ entries. Specializing further to permutation grids, we prove that every permutation grid admits at least one rook placement, and characterize the permutations whose grids admit exactly one placement in terms of the Robinson--Schensted correspondence. Throughout, we pose a variety of conjectures and open questions.
\end{abstract}



\maketitle

\section{Introduction}

Crossword puzzles are a popular pastime, a status they have held since the modern crossword puzzle made its debut in America in the \emph{New York World} in 1913 \cite{dau25}. American-style crossword puzzles are governed by a short, explicit set of construction rules---there is an $n \times n$ grid of white squares and black squares with $180^\circ$ rotational symmetry, the white squares form a single connected region, and all of the ``words" in the grid are made up of at least three white squares. We call a grid which satisfies the conditions above a \emph{legal} crossword grid. 

The structure underlying crossword puzzles raises a number of interesting mathematical questions, and there is a growing body of mathematical research about crosswords. For example, Ferry showed that when $n = 15$ (the most common crossword grid size), the number of legal grids such that there are no all-black rows or columns is  404,139,015,237,875 \cite{ferry}. For general $n$, the number of legal grids is not known. Ferland and Pratt studied the maximum number of words that can appear in a legal crossword grid \cite{fp19, f20}. They prove upper and lower bounds, determining the exact maximum for all odd $n$. 

In addition to combinatorial questions about crossword grids, there has also been work studying \emph{filling} crossword grids and solving crossword puzzles. Gourv\`es, Harutyunyan, Lampis, and Melissinos showed that the problem of filling a crossword grid with letters so that every across word and down word belongs to some fixed dictionary is NP-hard \cite{ghlm}. McSweeney studied how the structure of a crossword grid influences solving difficulty by modeling crossword solving as an epidemic-type process on a certain bipartite graph associated to the grid \cite{m16} (see the definition of $\Graph{G}$ below), while Hartmann modeled crossword-solving as a percolation process \cite{hart}.

This paper is inspired by a crossword puzzle by Marshal Herrmann that appeared in \emph{The New York Times} on January 16, 2024 \cite{puzzle}. In the solution to this crossword puzzle, pictured in Figure~\ref{fig:puzz}, the letter ``L" appears exactly once in every across word and every down word.  This is strongly reminiscent of a \emph{non-attacking rook placement} on the grid, where each ``L" is interpreted as a rook attacking every square in the words containing it.  

In the classical notion of a rook placement on a chessboard (as in, e.g., \cite{gjw, kr, r58}), a placement of rooks on a board $B$ (a subset of the rectangular grid $[m] \times [n]$) is \emph{non-attacking} if no two rooks share a row or column.  The notion of ``rook placements on a crossword grid'' differs in that the rooks cannot jump over black squares (i.e., those not belonging to the board $B$).  Although its roots are recreational, classical rook theory turns out to have deep connections to many areas of mathematics.  In this paper, we hope to show that ``crossword rook theory'' likewise has rich connections to modern discrete mathematics.  For example, we establish that rook placements on certain (\emph{sparse}) crossword grids are precisely the same as alternating sign matrices, and give a new characterization of the \emph{skew-merged permutations} in terms of rook placements on crossword boards. 

\begin{figure}
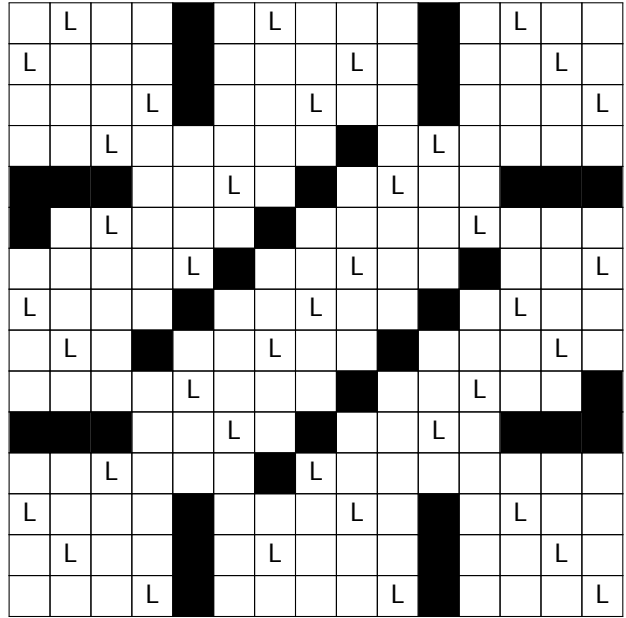

\label{fig:puzz}
\ytableausetup{boxsize=1.5em}
\[
\begin{ytableau}
*(white)&\LL&&&*(black)&&\LL&&&&*(black)&&\LL&& \\
\LL&&&&*(black)&&&&\LL&&*(black)&&&\LL& \\
&&&\LL&*(black)&&&\LL&&&*(black)&&&&\LL \\
&&\LL&&&&&&*(black)&&\LL&&&& \\
*(black)&*(black)&*(black)&&&\LL&&*(black)&&\LL&&&*(black)&*(black)&*(black)\\
*(black)&&\LL&&&&*(black)&&&&&\LL&&& \\
&&&&\LL&*(black)&&&\LL&&&*(black)&&&\LL \\
\LL&&&&*(black)&&&\LL&&&*(black)&&\LL&& \\
&\LL&&*(black)&&&\LL&&&*(black)&&&&\LL& \\
&&&&\LL&&&&*(black)&&&\LL&&&*(black) \\
*(black)&*(black)&*(black)&&&\LL&&*(black)&&&\LL&&*(black)&*(black)&*(black)\\
&&\LL&&&&*(black)&\LL&&&&&&& \\
\LL&&&&*(black)&&&&\LL&&*(black)&&\LL&& \\
&\LL&&&*(black)&&\LL&&&&*(black)&&&\LL& \\
&&&\LL&*(black)&&&&&\LL&*(black)&&&&\LL
\end{ytableau}
\]
\caption{
\emph{The New York Times} crossword puzzle grid from January 16, 2024, with every instance of the letter ``L" from the solution \cite{puzzle}.}
\end{figure}

The plan of the paper is as follows.  In Section~\ref{sec:generalgrids}, we introduce basic notation and definitions for rook placements on crossword grids and prove an upper bound for the number of rook placements that a crossword grid can admit. In Section~\ref{sec:permgrids}, we restrict our attention to sparse grids, that is, crossword grids where no two black squares share an edge. For certain sparse grids, we show a surprising connection to alternating sign matrices. We then focus on permutation grids, which have exactly one black square per row and column. In Section~\ref{sec:questions}, we pose questions and propose future research directions.

\subsection*{Acknowledgments}
J. B. Lewis was supported in part by a gift from the Simons Foundation (MPS-TSM-00006960).
R. Won was supported in part by Simons Foundation grant \#961085.

\section{Rook placements on general crossword grids}

\label{sec:generalgrids}

An $m \times n$ \emph{crossword grid} or \emph{grid} is an $m \times n$ array of squares, some of which are white and some of which are black. (We do not impose the symmetry, connectedness, or minimum word length assumptions that are the usual conventions for American-style crossword puzzles.)  To any crossword grid, we can associate an $m \times n$ matrix whose $(i,j)$ entry is $0$ if the corresponding square is white and $1$ if the corresponding square is black. We also refer to the squares in the grid by their row and column, using the matrix convention for numbering rows and columns. We refer to rows $1$ and $m$ as \emph{edge rows} and columns $1$ and $n$ as \emph{edge columns}. The squares in the edge rows and edge columns are called \emph{edge squares}. All non-edge squares in a crossword grid are called \emph{interior squares}. The squares $(1,1)$, $(1,n)$, $(m,1)$, and $(m,n)$ are called \emph{corner squares}.

We also use the cardinal directions to describe relative positions within a crossword grid. We say that the square $(i,j)$ is \emph{north} (N) of $(k, \ell)$ if $i < k$, \emph{south} (S) if $i > k$, \emph{west} (W) if $j < \ell$, and \emph{east} (E) if $j > \ell$. We also use the intercardinal directions: we say that $(i,j)$ is \emph{northwest} (NW) of $(k, \ell)$ if $i < k$ and $j < \ell$ and similarly for \emph{northeast} (NE), \emph{southwest} (SW), and \emph{southeast} (SE).

An \emph{across word} in a crossword grid is a maximal contiguous sequence of white squares within a single row. Similarly, a \emph{down word} is a maximal contiguous sequence of white squares within a single column. We say that an across word and down word \emph{intersect} if they share a common square.

To any crossword grid $G$, we can associate a bipartite graph $\Graph{G}$, as follows: the vertices in the first part correspond to the across words and the vertices in the second part correspond to the down words. There is an edge in $\Graph{G}$ between two words if and only if they intersect. This is the same bipartite graph defined in \cite{m16}; for other graphs that can be associated to crossword grids, see \cite{cm24}.

A \emph{(non-attacking) rook placement} on a crossword grid $G$ is a subset $R$ of the white squares in the grid such that no two squares in $R$ lie in the same across word or the same down word. We call the squares in $R$ \emph{rooks} and picture the rook placement by drawing $G$ and placing a chess rook into each square in $R$:
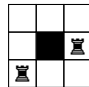
\[
\ytableausetup{boxsize=1em}
\begin{ytableau}
*(white)&&\\
&*(black)&\Rook\\
\Rook&&\end{ytableau}
\]
Each rook lies in one across word and one down word, and we say that the rook \emph{attacks} all of the other squares in its across word and its down word.

A rook placement $R$ on $G$ is called \emph{maximal} if for any rook placement $R'$ on $G$, if $R \subseteq R'$ then $R = R'$. We call $R$ \emph{maximum} if for any other rook placement $R'$ on $G$, we have $|R| \geq |R'|$. Finally, $R$ is called \emph{complete} if $R$ contains exactly one white square in every across word and every down word of $G$. We remark that no two of maximal, maximum, and complete coincide in general (see Figure~\ref{fig:maxcomp}) but complete implies maximum implies maximal. A rook placement on $G$ is equivalent to a matching on $\Graph{G}$, while a complete rook placement is equivalent to a perfect matching.
\begin{figure}[h!]
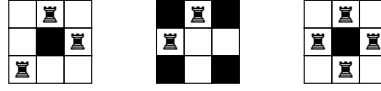

\label{fig:maxcomp}
\[
\begin{ytableau}
*(white)&\Rook&\\
&*(black)&\Rook\\
\Rook&&\end{ytableau}
\hspace{24pt}
\begin{ytableau}
*(black)&\Rook& *(black)\\
\Rook&&\\
*(black)&&*(black)\end{ytableau}
\hspace{24pt}
\begin{ytableau}
*(white)&\Rook&\\
\Rook&*(black)&\Rook\\
&\Rook&\end{ytableau}
\]
\caption{
For these three rook placements, the first is maximal but not maximum, the second is maximum but not complete, while the third is complete (and hence maximum and maximal).}
\end{figure}
Henceforth, throughout this paper, by a \emph{rook placement}, we mean a complete non-attacking rook placement. Given a grid $G$, let $\RP(G)$ denote the set of all possible (complete, non-attacking) rook placements on $G$. 

Given a grid $G$, we may also form the \emph{biadjacency matrix} $B_G$ of the bipartite graph $\Graph{G}$, which is the $\{0, 1\}$ matrix whose rows are indexed by across words and columns are indexed by down words and whose $(i,j)$-entry is $1$ if and only if the $i$th across and $j$th down word intersect. Observe that $|\RP(G)|$ is precisely  $\perm(B_G)$, the permanent of $B_G$.

Clearly, a necessary condition for a grid $G$ to admit a rook placement is that $G$ must have the same number of across words as down words, so that $\Graph{G}$ has the same number of vertices in each part. For example, the first grid below does not admit a rook placement since it has four across words and only three down words. However, this condition is not sufficient, as the second grid below has the same number of across words and down words (namely, three of each) but nevertheless does not admit a rook placement:
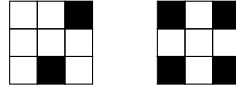
\begin{equation}\label{eq:no rook placements}
\begin{ytableau}
*(white)&&*(black)\\
&&\\
&*(black)&\\
\end{ytableau}
\hspace{24pt}
\begin{ytableau}
*(black)&& *(black)\\
&&\\
*(black)&&*(black)\end{ytableau}
\end{equation}

The examples above include words of length $1$ and $2$, and so violate the usual conventions of American-style crossword puzzles.  Our first result establishes that this requirement is superficial, in that the examples above can be extended to examples of the same behavior with any minimum word length. There is a simple \emph{$k$-inflation} operation on crossword grids in which each square (white or black) is replaced by a $k \times k$ block of the same color. For example, the $3$-inflation of the second example above is
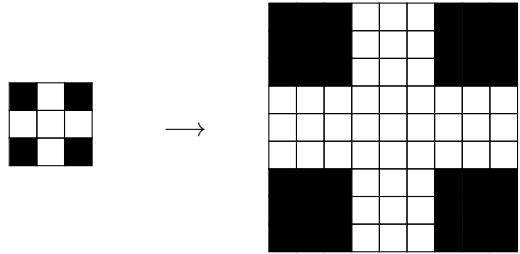
\begin{center}
\raisebox{-3em}{
\begin{ytableau}
*(black)&&*(black)\\
&& \\
*(black)&&*(black)
\end{ytableau}
}
\raisebox{-4em}{\qquad $\longrightarrow$\qquad}
\begin{ytableau}
*(black)&*(black)&*(black)&&&&*(black)&*(black)&*(black)\\
*(black)&*(black)&*(black)&&&&*(black)&*(black)&*(black)\\
*(black)&*(black)&*(black)&&&&*(black)&*(black)&*(black)\\
&&&&&&&& \\
&&&&&&&& \\
&&&&&&&& \\
*(black)&*(black)&*(black)&&&&*(black)&*(black)&*(black)\\
*(black)&*(black)&*(black)&&&&*(black)&*(black)&*(black)\\
*(black)&*(black)&*(black)&&&&*(black)&*(black)&*(black)
\end{ytableau}
\end{center}

\begin{lemma}\label{lem:inflation}
    A crossword grid admits a rook placement if and only if any of its $k$-inflations admits a rook placement.
\end{lemma}
\begin{proof}
 If a grid admits a rook placement then we can construct a rook placement on its $k$th inflation by replacing each white square that contains rook with a permutation matrix of rooks.  We now consider the converse.

 If a grid has bipartite graph $\Gamma$, then it is easy to describe the bipartite graph $\Gamma'$ of the $k$th inflation: replace each vertex $v$ of $\Gamma$ with $k$ vertices $v^1, \dots, v^k$ in $\Gamma'$. If there is an edge between $v$ and $w$ in $\Gamma$, then there are all possible edges between $\{v^1,\dots, v^k\}$ and $\{w^1, \dots, w^k\}$ in $\Gamma'$. If there is not an edge between $v$ and $w$ in $\Gamma$, then there are no edges between $\{v^1,\dots, v^k\}$ and $\{w^1, \dots, w^k\}$ in $\Gamma'$. 

 Suppose that a grid $G$ does not admit a rook placement, so that its bipartite graph $\Gamma$ does not admit a perfect matching.  By K\"{o}nig's theorem, there is some set of across vertices $v_1, \dots, v_n$ whose neighborhood has cardinality $N > n$. Retain the notation above for the bipartite graph for the $k$th inflation. In the inflation, the vertices $v_1^1, \dots v_1^k, v_2^1, \dots, v_2^k, \dots, v_n^1, \dots, v_n^k$ form a set of $kn$ vertices with neighborhood of cardinality $kN$, so the $k$th inflation does not admit a rook placement.
\end{proof}

In the remainder of this section, we address the following question: given a crossword grid $G$, how many rook placements does $G$ admit, that is, how large is $|\RP(G)|$? Since this question is equivalent to counting perfect matchings on $\Graph{G}$, in general, the problem is \#P complete. As a warm up, we consider the two easiest cases, restricting our attention to square $n \times n$ grids.

If $G$ has no black squares, then $|\RP(G)| = n!$. If $G$ has a single black square, then there are three cases: (1) the black square is a corner square, (2) the black square is an edge square but not a corner square, or (3) the black square is an interior square. In case (1), we have $|\RP(G)| = n! - (n-1)!$ as every rook placement is given by a permutation on the $n \times n$ grid which avoids the black corner. In case (2), the number of across words is not equal to the number of down squares so $|\RP(G)| = 0$. In case (3), suppose that the black square $B$ is in position $(i,j)$. 

In any rook placement on $G$, there must be a rook in the across word of length $j - 1$ to the west of $B$. Likewise, there are rooks in the across word of length $n - j$ to the east of $B$, the down word of length $i - 1$ to the north of $B$ and the down word of length $n - i$ to the south of $B$. For any choice of placement for these four rooks, if we delete every square that is attacked in $G$, we are left with an $(n - 3) \times (n - 3)$ grid of white squares, which admits $(n - 3)!$ possible rook placements. Hence
\[
|\RP(G)| = (i-1)(j-1)(n-i)(n-j) \cdot (n-3)!.
\]
This is maximized when $i = j = \lceil n/2 \rceil$.  Thus, as $n$ grows, the maximum number of rook placements with a single black square grows asymptotically to $\frac{n}{16} \cdot n!$, when the single black square is as close to the center as possible.

\begin{question}
\label{q:2mostperk}
 For a fixed number $k$ of black squares in an $n \times n$ grid, which arrangement admits the most rook placements? How many?
\end{question}

Of course, specifying that an $m \times n$ grid has $k$ black squares is the same as specifying that it has $mn - k$ white squares. We prove the following universal upper bound on the number of rook placements that a grid can admit depending on the number of white squares and across words in the grid.

\begin{proposition}\label{prop:upper bound}
If $G$ is a crossword grid with $N$ white squares and $a$ across words, then 
\[
|\RP(G)| \leq \left[ \left(\frac{N}{a}\right)^{1 + a/N} e^{a/N - 1}\right]^a  \leq e^{CN},
\]
where $C \approx 0.27386$ is the unique maximum value of $f(x) = - x \ln(x) - x^2 \ln(x) - x + x^2$ on the interval $(0,1)$.
\end{proposition}

The same bounds hold if we use $d$, the number of down words, in place of $a$.

\begin{proof}
Let $G$ be a grid with $a$ across words and $N$ white squares. Since $|\RP(G)| = \perm(B_G)$ where $B_G$ is the biadjacency matrix of $\Graph{G}$, the Bregman--Minc inequality \cite{bregman} gives the upper bound
\[
|\RP(G)| = \perm(B_G) \leq \prod_{i=1}^a (r_i!)^{1/r_i},  
\]
where $r_i$ is the $i$th rowsum of $B_G$. Since the $i$th row of $B_G$ has a $1$ for each down word which intersects the $i$th across word, $r_i$ is precisely the length of the $i$th across word. We therefore obtain
\[
|\RP(G)| \leq \prod_{b \text{ an across word}} (\len(b)!)^{1/\len(b)}.
\]

For each positive integer $r$, a standard bound (coming from a left-endpoint Riemann sum approximation for $\int_{1}^{r} \ln(x) \, dx$) is $\ln(r!) \leq r \ln r - r + \ln r + 1$. Writing $g(r) = \frac{1}{r} \ln(r!)$ and $\phi(r) = \ln r - 1 + \frac{1 + \ln r}{r}$, this means that $g(r) \leq \phi(r)$ for all positive integers $r$. 

The second derivative $\phi''(r) = \frac{2 \ln (r) - r - 1}{r^3}$ is negative for $r > 0$ and so $\phi(r)$ is concave down on the interval $(0,\infty)$. By Jensen's inequality, we have
\[
\ln |\RP(G)| \leq \sum_{i = 1}^a g(r_i) \leq \sum_{i=1}^a \phi(r_i) \leq a \phi\left( \frac{\sum r_i}{a} \right).
\]
Moreover, since $\sum_{i=1}^{a} r_i = N$ is the number of white squares, we see that
\[
\ln |\RP(G)| \leq a \phi\left(\frac{N}{a}\right) = a \ln\left(\frac{N}{a} \right) - a + \frac{a^2}{N} + \frac{a^2}{N}\ln\left(\frac{N}{a} \right) 
\]
and so
\[
|\RP(G)| \leq \left[ \left(\frac{N}{a}\right)^{1 + a/N} e^{a/N - 1}\right]^a.
\]

Let $f(x) = - x \ln(x) - x^2 \ln(x) - x + x^2$ and observe that $Nf(a/N) = a \phi(N/a)$. Since $a/N \in (0, 1]$, we have that $\ln|\RP(G)| \leq N f(x^*)$ where $x^* \in (0,1]$ maximizes $f(x)$. Since $\lim_{x \to 0^+} f(x) = 0 = f(1)$, the maximum of $f(x)$ on this interval occurs at an interior critical point $x^*$. Differentiating yields $f'(x) = - \ln(x) - 2x \ln(x) + x - 2$ and setting $f'(x) = 0$ shows that 
\[
(1 + 2x^*) \ln(x^*) = x^* - 2.
\]
Numerically solving for $x^*$ gives a unique $x^* \approx 0.42987$ with $f(x^*) \approx 0.27386$. This completes the proof.
\end{proof}

For $k$ in particular regimes, it should be possible to give much sharper bounds.  For example, we conjecture the following bounds when the number of black squares $k$ is small compared to $n$.

\begin{conj}
Let $G$ be an $n \times n$ crossword grid with $k$ black squares. If $k = o(n)$ then
    \[
    |\RP(G)| \leq \frac{n^{4k}}{2^{4k}} \cdot (n - 3k)!,
    \]
and this bound is achieved asymptotically for some $G$.
\end{conj}

We close this section by giving an example of a crossword grid which admits many placements.

\begin{example}
For $n = 6m + 3$, consider the grid
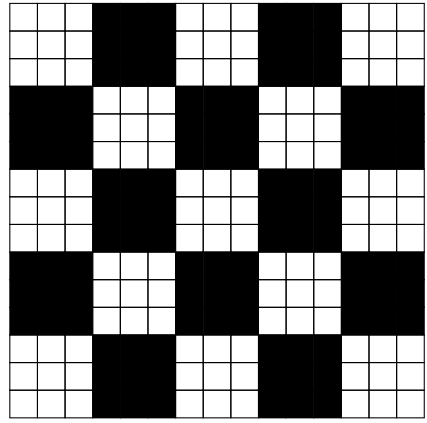
\[
\begin{ytableau}
*(white)&&&*(black)&*(black)&*(black)&&&&*(black)&*(black)&*(black)&&& \\
&&&*(black)&*(black)&*(black)&&&&*(black)&*(black)&*(black)&&& \\
&&&*(black)&*(black)&*(black)&&&&*(black)&*(black)&*(black)&&& \\
*(black)&*(black)&*(black)&&&&*(black)&*(black)&*(black)&&&&*(black)&*(black)&*(black)\\
*(black)&*(black)&*(black)&&&&*(black)&*(black)&*(black)&&&&*(black)&*(black)&*(black)\\
*(black)&*(black)&*(black)&&&&*(black)&*(black)&*(black)&&&&*(black)&*(black)&*(black) \\
&&&*(black)&*(black)&*(black)&&&&*(black)&*(black)&*(black)&&& \\
&&&*(black)&*(black)&*(black)&&&&*(black)&*(black)&*(black)&&& \\
&&&*(black)&*(black)&*(black)&&&&*(black)&*(black)&*(black)&&& \\
*(black)&*(black)&*(black)&&&&*(black)&*(black)&*(black)&&&&*(black)&*(black)&*(black)\\
*(black)&*(black)&*(black)&&&&*(black)&*(black)&*(black)&&&&*(black)&*(black)&*(black)\\
*(black)&*(black)&*(black)&&&&*(black)&*(black)&*(black)&&&&*(black)&*(black)&*(black) \\
&&&*(black)&*(black)&*(black)&&&&*(black)&*(black)&*(black)&&& \\
&&&*(black)&*(black)&*(black)&&&&*(black)&*(black)&*(black)&&& \\
&&&*(black)&*(black)&*(black)&&&&*(black)&*(black)&*(black)&&&
\end{ytableau}
\]
made of $(2m + 1)^2$ $3 \times 3$ blocks, whose density of black squares is just slightly less than $\frac{1}{2}$.  There are
\[
6^{2m^2 + 2m + 1} = 6^{n^2/18 + 1/2}
\]
rook placements on this grid.
Consider the connected grid that arises by cutting the corners out of each black block:
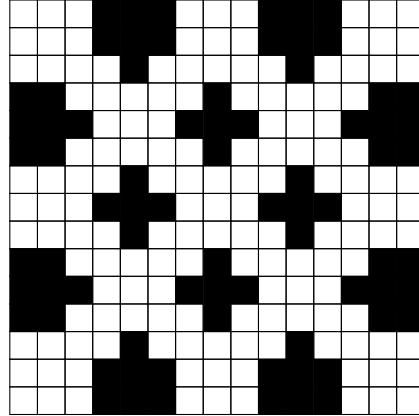
\[
\begin{ytableau}
*(white)&&&*(black)&*(black)&*(black)&&&&*(black)&*(black)&*(black)&&& \\
&&&*(black)&*(black)&*(black)&&&&*(black)&*(black)&*(black)&&& \\
&&&&*(black)&&&&&&*(black)&&&& \\
*(black)&*(black)&&&&&&*(black)&&&&&&*(black)&*(black)\\
*(black)&*(black)&*(black)&&&&*(black)&*(black)&*(black)&&&&*(black)&*(black)&*(black) \\
*(black)&*(black)&&&&&&*(black)&&&&&&*(black)&*(black)\\
&&&&*(black)&&&&&&*(black)&&&& \\
&&&*(black)&*(black)&*(black)&&&&*(black)&*(black)&*(black)&&& \\
&&&&*(black)&&&&&&*(black)&&&& \\
*(black)&*(black)&&&&&&*(black)&&&&&&*(black)&*(black)\\
*(black)&*(black)&*(black)&&&&*(black)&*(black)&*(black)&&&&*(black)&*(black)&*(black)\\
*(black)&*(black)&&&&&&*(black)&&&&&&*(black)&*(black)\\
&&&&*(black)&&&&&&*(black)&&&& \\
&&&*(black)&*(black)&*(black)&&&&*(black)&*(black)&*(black)&&& \\
&&&*(black)&*(black)&*(black)&&&&*(black)&*(black)&*(black)&&&
\end{ytableau}
\]
Observe that every word in this grid contains (as a subset) one word in the previous grid, some squares that were black in the previous grid, and nothing else; consequently each rook placement on the previous grid is also a rook placement on this grid. 

In the case $n = 6m + 3$, the number of white squares in this connected grid is $N = 9(2m^2 + 2m + 1) + 2(4m) + 4(2m^2 - 2m) = 26m^2 + 18m + 9 = \frac{13}{18}n^2 - \frac{4}{3}n + \frac{13}{2}$, and it admits at least $C \cdot 6^{n^2/18} \approx C \cdot \exp(0.09954n^2)$ rook placements, whereas the upper bound of Proposition~\ref{prop:upper bound} is less than $\exp(.274 \cdot (\frac{13}{18}n^2 - \text{l.o.t.})) \approx C' \cdot \exp(.198 n^2)$. 
\end{example}

The comparison between Proposition~\ref{prop:upper bound} in the previous example raises the following questions. 
\begin{question}
\label{q:maxasymptotics}
Let $\kappa$ denote the density of black squares in an $n \times n$ crossword grid. For each density $\kappa$, as $G$ ranges over all $n \times n$ grids with black square density $\kappa$, does $\displaystyle \lim_{n\to\infty} n^{-2}\ln \left(\max_G |\RP(G)|\right)$ exist? When it exists, what is its value?

Which $n \times n$ grid $G$ maximizes the number of rook placements? How many rook placements does this $G$ admit, and what is the density of black squares in $G$?
\end{question}

In the remainder of the paper, we turn our attention from general grids to \emph{sparse} grids.

\section{Sparse grids and permutation grids}
\label{sec:permgrids}

In this section, we consider crossword grids in which no two black squares share an edge.  We call such a crossword grid \emph{sparse}.  We first show a surprising connection between rook placements on sparse crossword grids and \emph{alternating sign matrices}.  Afterwards, we focus on the natural family of \emph{permutation grids}, which are sparse grids with exactly one black square in each row and column.  We completely characterize the permutations whose grids admit the minimum number of rook placements (a subset of the set of \emph{skew-merged permutations}) and formulate a conjecture about the permutations whose grids admit the maximum number of placements.

\subsection{Sparse grids and alternating sign matrices}
\label{subsec:asms}

An \emph{alternating sign matrix} (ASM) is a $\{0, +1, -1\}$-matrix in which every row and column sums to $1$ and the nonzero entries in each row and column alternate in sign.  Alternating sign matrices arise naturally in algebra, statistical mechanics, and representation theory, and admit bijections with a rich variety of combinatorial structures \cite{eklp92, kup02, mrr83, zei96, LascouxSchutzenberger, StrikerASMpolytope, FK1, FK2}; for broad overviews at various levels of detail, see \cite{Bressoud, pro01, StrikerSurvey}.

\begin{figure}[h!]
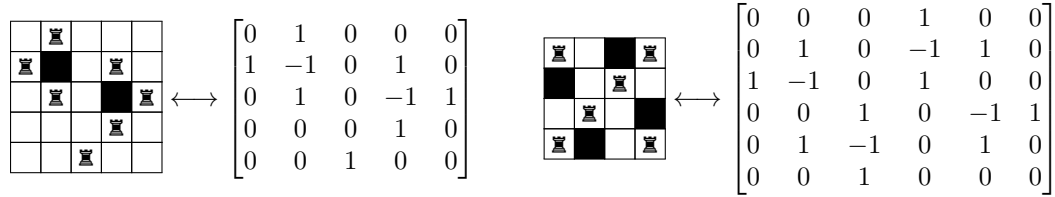

\ytableausetup{boxsize=1.1em}
\[
  \raisebox{1.8\height}{\begin{ytableau}
*(white)&\Rook&&&  \\
\Rook&*(black)&&\Rook&\\
&\Rook&&*(black)&\Rook\\
&&&\Rook& \\
&&\Rook&&
\end{ytableau}}
\longleftrightarrow
\begin{bmatrix}
    0 & 1 & 0 & 0 & 0 \\ 
    1 & -1 & 0 & 1 & 0 \\
    0 & 1 & 0 & -1 & 1 \\
    0 & 0 & 0 & 1 & 0 \\
    0 & 0 & 1 & 0 & 0
\end{bmatrix}
\hspace{24pt}
  \raisebox{1.25\height}{\begin{ytableau}
*(white)\Rook & & *(black) &\Rook\\
*(black)&&\Rook&\\
&\Rook&&*(black)\\
\Rook&*(black)&&\Rook
\end{ytableau}}
\longleftrightarrow
\begin{bmatrix}
0 & 0 & 0 & 1 & 0 & 0 \\
0 & 1 & 0 & -1 & 1 & 0 \\
1 & -1 & 0 & 1 & 0 & 0 \\
0 & 0 & 1 & 0 & -1 & 1 \\
0 & 1 & -1 & 0 & 1 & 0 \\
0 & 0 & 1 & 0 & 0 & 0
\end{bmatrix}
\]
    \caption{Rook placements on sparse crossword grids, and their corresponding ASMs.}
    \label{fig:asms}
\end{figure}

It turns out that rook placements on sparse grids are nearly the same object as ASMs.  We make this precise in the following statement, which is illustrated in the left half of Figure~\ref{fig:asms}.

\begin{proposition}
\label{prop:ASMs}
Suppose that $G$ is a sparse $n \times n$ crossword grid with no black squares on the edges.  Then there is a bijection between rook placements on $G$ and $n \times n$ alternating sign matrices with $-1$s whose positions are the same as the positions of the black squares in $G$.
\end{proposition}
\begin{proof}
Given an ASM $A$, construct a crossword grid of the same size by replacing every $-1$ with a black square and every $0$ and $+1$ with a white square.  Place rooks on this grid at the positions of the $+1$s in $A$.  By the definition of ASMs, there is exactly one rook between each pair of consecutive black squares in every row and column, and exactly one rook north of the northernmost black square in each column (respectively south of the southernmost, west of the westernmost in each row, and east of the easternmost in each row).  Thus this placement of rooks is nonattacking and complete.  The inverse map is equally straightforward to describe.
\end{proof}

\begin{remark}
\label{rem:asm-perm}
One can extend Proposition~\ref{prop:ASMs} to also give correspondences between ASMs and rook placements on other families of sparse grids.  For example, if $G$ is an $n \times n$ permutation grid (or, more generally, a sparse grid with exactly one black square on each of the four edges), then there is a map from rook placements on $G$ to $(n + 2) \times (n + 2)$ ASMs that comes from padding $G$ with new edge rows and columns containing no black squares; the position of the $+1$ in each of the edges is forced by the $-1$.  (This is illustrated in the right half of Figure~\ref{fig:asms}.)

For sparse grids with multiple black squares on the edges, one could pad by multiple new edge rows that contain no black squares; in general, however, this will not give a bijective relationship with ASMs.
\end{remark}

The number of $n \times n$ ASMs has a beautiful product formula, and many interesting subclasses of ASMs have also been enumerated.  However, as far as the present authors are aware, essentially nothing is known about the enumeration of ASMs according to the positions of their $-1$s. 
In the rest of this section, although we drop the language of ASMs, our results can also be interpreted (via Proposition~\ref{prop:ASMs} and Remark~\ref{rem:asm-perm}) as enumerations of certain classes of ASMs.

\subsection{Permutation grids}
\label{subsec:perm}

We now turn our focus specifically to permutation grids, that is, grids with exactly one black square in each row and column.
We adopt the convention that if a permutation matrix has a $1$ in position $(i,j)$ then the corresponding permutation maps $i$ to $j$. Equivalently, for a permutation $w \in S_n$, we define the crossword grid $\Grid{w}$ to be the grid with black squares in $(i, w(i))$ for all $i \in [n]$.

It is easy to see that if an $n \times n$ permutation grid admits a rook placement, then this rook placement will have $2n - 2$ rooks, since each black square other than those in the first and last columns divides its row into two across words (its \emph{west} and \emph{east across word}), and likewise each black square other than those in the first and last rows divides its column into two down words.
Next, we show that every permutation grid admits at least one rook placement.

\begin{proposition}\label{prop:existence}
    Every $n \times n$ permutation grid admits a rook placement.
\end{proposition}
\begin{proof}
    The result is clearly true for $n  = 1$.  Suppose it is true for $n = k$, and consider a $(k + 1) \times (k + 1)$ grid $G$ with exactly one black square in every row and column. In particular, there must be a black square in the easternmost column, say in position $(i, k + 1)$.  Delete column $k + 1$ and row $i$ from $G$; the result is a $k \times k$ grid $G'$ with exactly one black square in each row and column.  By the inductive hypothesis, $G'$ admits a rook placement, using exactly $2k - 2$ rooks. Use this placement to construct a (nonattacking, but incomplete) placement of $2k - 2$ rooks on $G$ in the obvious way (in rows $1, \ldots, i - 1, i + 1, \ldots, k + 1$ and columns $1, \ldots, k$).  We have four cases:

    First, if $i = k + 1$, choose $j, j'$ so that there are black squares in positions $(k, j)$ and $(j', k)$. We illustrate an example with these two black squares highlighted in red.
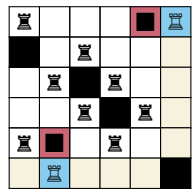
\[
\begin{ytableau}
*(white)\Rook&&&& *(color1)\blacksquarenew &*(color3) \WRook  \\
*(black)&&\Rook&&&*(lc2)\\
&\Rook&*(black)&\Rook&&*(lc2)\\
&&\Rook&*(black)&\Rook&*(lc2)\\
\Rook&*(color1)\blacksquarenew &&\Rook&&*(lc2)\\
*(lc2)&*(color3)\WRook &*(lc2)&*(lc2)&*(lc2)&*(black)
\end{ytableau}
\]
Then in $G$, there are two down words and two across words missing rooks (the across word to the east of $(j', k)$, the down word to the south of $(k,j)$, the across word to the west of $(k+1, k +1)$, and the down word to the north of $(k+1, k+1)$). We can simply add rooks in positions $(k + 1, j)$ and $(j', k + 1)$ (highlighted in blue in our illustration) to complete the rook placement on $G$.

    Second, if $i = 1$ we have the mirror image of the previous case.

    Third, if $1 < i < k + 1$ and the black square in column $k$ is in position $(j, k)$ with $j < i$, then do the following: consider the east across word in row $i + 1$.  Since $j \neq i + 1$, this word has length larger than $1$ and has a rook in position $(i + 1, j')$ for some $j'$ (highlighted in orange in our illustration).
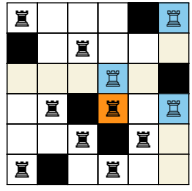
\[
\begin{ytableau}
*(white)\Rook&&&& *(black)&*(color3)\WRook   \\
*(black)&&\Rook&&&*(lc2)\\
*(lc2)&*(lc2) &*(lc2)&*(color3)\WRook&*(lc2)&*(black)\\
&\Rook&*(black)&*(orange)\Rook&&*(color3) \WRook\\
&&\Rook&*(black)&\Rook&*(lc2)\\
\Rook&*(black)&&\Rook&&*(lc2)
\end{ytableau}
\] 
Remove this rook; then add rooks in positions $(i, j')$, $(i+1, k + 1)$, and $(j, k + 1)$ (highlighted in blue in our illustration).  This gives $2k = 2(k+1) - 2$ rooks for a rook placement on $G$.

Finally, the last case (with $1 < i < k + 1$ and a black square in position $(j, k)$ with $j > i$) is the mirror image of the third case.
\end{proof}

\begin{remark}
\label{rem.injection}
    The proof of the previous proposition actually gives an injection from rook placements on the $k\times k$ grid $G'$ to rook placements on the $(k + 1)\times(k + 1)$ grid $G$. In particular, if a $(k+1) \times (k+1)$ grid admits $i$ rook placements, then the $k \times k$ grid obtained by deleting the easternmost black square's row and column admits no more than $i$ rook placements.
\end{remark}

Proposition~\ref{prop:existence} can be generalized slightly, to the family of \emph{partial permutation grids}, which contain at most one black square in each row and column.  (The second example in \eqref{eq:no rook placements} shows that it cannot be extended to all sparse grids.)

\begin{corollary}
\label{cor:partialperm}
    Suppose $G$ is an $n \times n$ partial permutation grid. Then $G$ admits a rook placement if and only if the number of across words in $G$ is equal to the number of down words in $G$, if and only if the total number of black squares in columns $1$ and $n$ is equal to the total number of black squares in rows $1$ and $n$.
\end{corollary}

\begin{proof}
    First note that since $G$ is a partial permutation grid, it is sparse. If we begin with an empty $n \times n$ grid and add  black squares one at a time until we reach $G$, then each black square in the interior of the grid increases the number of across words and the number of down words by one. On the other hand, each black square in column $1$ and $n$ adds a down word but does not add an across word. Therefore, the condition that the number of across words is equal to the number of down words is equivalent to the condition that the total number of black squares in columns $1$ and $n$ equals to the total number of black squares in rows $1$ and $n$.
    In order for $G$ to admit a rook placement, it is clearly necessary that the number of across words is equal to the number of down words. Hence, we need only show the converse.

    Since there is at most one black square in every row and column of $G$, the number of rows containing no black squares is equal to the number of columns containing no black squares. Call such a row or column \emph{all-white}. Ignoring these all-white rows and columns yields a subgrid which is a permutation matrix, and which, by Proposition~\ref{prop:existence}, admits a rook placement. Place rooks in the corresponding squares on $G$.

    Since $G$ has the same number of black squares in columns $1$ and $n$ as in rows $1$ and $n$, we see that the number of all-white edge rows is equal to the number of all-white edge columns.

    Now add rooks to the all-white rows and columns as follows. If there is an all-white edge row, then there is a corresponding all-white edge column. Add these rows to the permutation submatrix and place a rook in the two new length-one words that appear. If the grid has an additional all-white edge row and column, also add these  rows to the previous submatrix and place two more rooks in the new length-one words. Now every remaining all-white row and all-white column appears in the interior of the previous submatrix. Pair each row with a column and place a rook in the intersection. This gives a rook placement on $G$.
\end{proof}

\subsection{Permutation grids admitting small numbers of placements}
\label{subsec:smallnumber}

The existence proof above establishes that every permutation grid admits at least one rook placement.  This naturally raises the question of which permutation grids admit \emph{exactly} one rook placement (the minimum possible number).  We answer this question now, after a preparatory lemma.

\begin{lemma}
\label{lem.mutualrooks}
    Suppose that a non-attacking rook placement on a crossword grid has two rooks $R_1$ and $R_2$ that both attack the squares $x_1$ and $x_2$. Then removing the rooks $R_1$ and $R_2$ and placing rooks in $x_1$ and $x_2$ yields another non-attacking rook placement on the crossword grid.
\end{lemma}
For example, in the grid below, the rooks in positions $(3,4)$ and $(4,2)$ (highlighted in orange) mutually attack the blue squares $(3,2)$ and $(4,4)$ and so removing those rooks and placing rooks on the blue squares yields another rook placement on the grid.
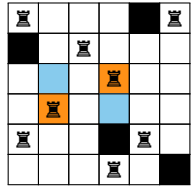
\[
\begin{ytableau}
*(white)\Rook&&&& *(black)& \Rook \\
*(black)&&\Rook&&&\\
&*(color3)&&*(orange)\Rook&&\\
&*(orange)\Rook&&*(color3)&&\\
\Rook&&&*(black)&\Rook&\\
&&&\Rook&&*(black)
\end{ytableau}
\]
\begin{proof}[Proof of Lemma~\ref{lem.mutualrooks}]
    The non-attacking rook placement on the crossword grid corresponds to a matching in the bipartite graph, where each rook corresponds to an edge in the matching. The squares that each rook attacks correspond to edges which share a vertex with the rook's edge. If two rooks mutually attack two squares, then this gives a copy of $K_{2,2}$ in the bipartite graph. Choosing the other two edges in the copy of $K_{2,2}$ yields another perfect matching and hence another non-attacking rook placement.
\end{proof}

We now characterize the permutation grids that admit a unique rook placement.

\begin{theorem}\label{thm:one placement}
Let $w \in S_n$. Then $\Grid{w}$ admits exactly one rook placement if and only if $w$ is the union of an increasing subsequence and decreasing subsequence that overlap.
\end{theorem}
\begin{proof}
Any pair of a (strictly) increasing and a decreasing sequence can overlap in at most one element.  Therefore, if $w \in S_n$ is the union of an increasing sequence and decreasing sequence that overlap, there is an element $c \in [n]$ such that there is an increasing subsequence $a_1 < \dots < a_i < c < b_1 < \dots < b_j$ and decreasing subsequence $d_1 > \dots > d_k > c > e_1 > \dots > e_{\ell}$ of $w$ and $[n]$ is a disjoint union of the $a$'s, $b$'s, $d$'s, $e$'s, and $c$.  It is possible that there are multiple choices of this decomposition, but if so, we fix one once and for all.  It is also possible that some of the indices $i,j,k, \ell$ are equal to $0$. For $m \in [n]$, let $B_m$ denote the black square in column $m$. 

As a running example, we consider the permutation $5~1~3~4~2~6$ in $S_6$ with the increasing subsequence $1 < 3 < 4 < 6$ and the decreasing subsequence $5 > 4 > 2$, so that $c = 4$.  (We could have also chosen the decreasing subsequence $5 > 3 >2$.)  We illustrate the grid below with the overlapping square highlighted in red:
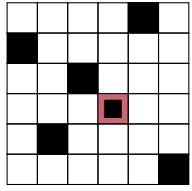
\[
\begin{ytableau}
*(white)&&&& *(black)&  \\
*(black)&&&&&\\
&&*(black)&&&\\
&&&*(color1)\blacksquarenew&&\\
&*(black)&&&&\\
&&&&&*(black)
\end{ytableau}
\]

In $\Grid{w}$, the black squares that occur to the west of $B_c$ are exactly $B_{a_1}, \dots, B_{a_i}$ and $B_{e_1}, \dots, B_{e_{\ell}}$, and we have listed these black squares in order from north to south. Therefore, the west across words of length less than or equal to $c - 1$ occur directly to the west of these black squares.
    
Now observe that from north to south, these west across words first increase monotonically in length and then decrease monotonically in length. Hence, any rook in the west across word of length $m$ with $1 \leq m \leq c - 1$ attacks a square in every west across word of length $M$ where $m < M \leq c - 1$. In any rook placement on this grid, there is a rook in the unique white square of the west across word of length $1$ (which occurs to the west of $B_2$). Since this rook attacks a square (the left square) in the west across word of length $2$, the rook in this across word must be in the easternmost square of this word. That is, there must be a rook directly to the west of $B_3$. Similarly, since these two rooks attack the west across word of length $3$, there must be a rook directly to the west of $B_4$. By induction, each of the black squares $B_{a_1}, \dots, B_{a_i}, B_{c}, B_{e_1}, \dots, B_{e_{\ell}}$ other than $B_1$ must have a rook directly to its west.  In our running example, this establishes that the following arrangement is forced in any rook placement:
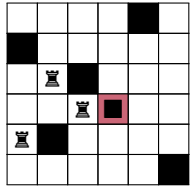
\[
\begin{ytableau}
*(white)&&&& *(black)&  \\
*(black)&&&&&\\
&\Rook&*(black)&&&\\
&&\Rook&*(color1)\blacksquarenew&&\\
\Rook&*(black)&&&&\\
&&&&&*(black)
\end{ytableau}
\]

By a symmetric argument (rotating the grid by $90^\circ$ or $180^\circ$), in any rook placement there must be a rook directly to the east of the black squares $B_{d_1}, \dots, B_{d_k}, B_{c}, B_{b_1}, \dots, B_{b_j}$ other than $B_n$, directly to the north of the black squares $B_{a_1}, \dots, B_{a_i}, B_{c}, B_{d_1}, \dots, B_{d_k}$ other than the black square in row $1$, and directly to the south of the black squares $B_{b_1}, \dots B_{b_k}, B_{c}, B_{e_1}, \dots, B_{e_{\ell}}$ other than the black square in row $n$. Continuing with our running example, the following rooks are forced in any rook placement:
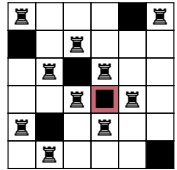
\[
\ytableausetup{boxsize=1em}
\begin{ytableau}
*(white) \Rook &&&& *(black)&\Rook  \\
*(black)&&\Rook&&&\\
&\Rook&*(black)&\Rook&&\\
&&\Rook&*(color1)\blacksquarenew&\Rook&\\
\Rook&*(black)&&\Rook&&\\
&\Rook&&&&*(black)
\end{ytableau}
\]

In general, this means that in any rook placement on $\Grid{w}$, there are four rooks placed adjacent to $B_c$ and for each of the remaining $n - 1$ black squares, two additional rooks placed adjacent those black squares (other than four rooks). It is easy to see from the structure of $w$ that these rooks are all different from each other, so altogether, this gives $4 + 2(n-1) - 4 = 2n - 2$ rooks, which is the total number of rooks in the rook placement. Hence, there is a unique rook placement on the grid.

For the converse, we proceed by induction on $n$. Suppose that for all $v \in S_{n-1}$, if $\Grid{v}$ admits a unique rook placement, then $v$ is the union of an increasing and decreasing subsequence which overlap.  
Suppose that $w \in S_n$ is \emph{not} the union of an increasing and decreasing subsequence that overlap; we will show that $\Grid{w}$ admits more than one rook placement. As above, for $m \in [n]$, let $B_m$ denote the black square in column $m$ of $\Grid{w}$.

Consider the permutation $w' \in S_{n - 1}$ whose one-line notation is the result of deleting $n$ from the one-line notation of $w$; in terms of the grid, $\Grid{w'}$ is obtained by deleting the $n$th column of $\Grid{w}$ and the row which contains $B_n$.  If $\Grid{w'}$ admits more than one rook placement, then by Remark~\ref{rem.injection}, $\Grid{w}$ admits more than one rook placement and we are done.  So, we may assume that $\Grid{w'}$ admits a unique rook placement. By the induction hypothesis, $w'$ is the union of an increasing and decreasing subsequence that overlap.  
Fix the choice of decomposition of $w'$ into subsequences such that their overlapping element $c$ is as large as possible. The overlapping black square $B_c$ divides $\Grid{w'}$ into quadrants which are NW, NE, SW, and SE of $B_c$. Each other black square lies in one of these quadrants; the decreasing subsequence corresponds to the black squares in the NE and SW quadrants while the increasing subsequence corresponds to the black squares in the NW and SE quadrants.

Without loss of generality, we may assume that $B_n$ occurs to the south of $B_c$. Note that $B_n$ cannot be south of the southernmost black square in the increasing subsequence of $w'$, as otherwise $w$ would be the union of an increasing and decreasing subsequence which overlap, and this would contradict our assumption. Hence, there must be at least one black square in the SE quadrant which is south of $B_n$.

 In any rook placement on $\Grid{w}$, there must be a rook directly to the east of $B_{n-1}$. We now divide into two cases, depending on whether $B_{n - 1}$ occurs in the NE or SE quadrant.

\textbf{Case 1:} Suppose that $B_{n-1}$ occurs in the NE quadrant. As in the proof of Proposition~\ref{prop:existence}, each black square in the SE quadrant which occurs south of $B_n$ will give a rook placement on $\Grid{w}$ by extending from the unique rook placement on $\Grid{w'}$; indeed, in the (not complete) rook placement on $\Grid{w}$ inherited from the placement on $\Grid{w'}$, each such black square has a rook directly to its east that attacks both a square in the across word to the west of $B_n$ and a square in the down word to the south of $B_n$. Hence, if there are at least two black squares in the SE quadrant which are south of $B_n$, we obtain at least two rook placements on $\Grid{w}$, and we are done.

Therefore, we may assume that there is exactly one black square, say $B_i$, to the SE of $B_c$ and south of $B_n$. Since $B_{n - 1}$ is north of $B_c$, we have $i \neq n-1$. In the unique rook placement on $\Grid{w'}$, the rook $R$ to the east of $B_i$ yields a rook placement on $\Grid{w}$, as in Proposition~\ref{prop:existence}, where we replace $R$ with rooks in the squares which $R$ attacks in the across word to the west of $B_n$ and the down word to the south of $B_n$. In this rook placement on $\Grid{w}$, let $R_1$ be the rook placed in the across word to the west of $B_n$. Consider the black square $B_j$ which is the second furthest south in the SE quadrant (this may be $B_c$). Let $R_2$ be the rook to the east of $B_j$. Then $R_1$ and $R_2$ mutually attack two squares. By Lemma~\ref{lem.mutualrooks}, $\Grid{w}$ therefore admits a second rook placement.

\textbf{Case 2:} Suppose that $B_{n-1}$ occurs in the SE quadrant. Consider the furthest south black square $B_j$ in the SE quadrant that is to the north of $B_{n}$ (this may be the overlapping square $B_c$). In the unique rook placement on $\Grid{w'}$, there is a rook $R$ directly to the east of $B_{j}$. In $\Grid{w}$, $R$ attacks both the words to the north and to the west of $B_n$, so one rook placement on $\Grid{w}$ arises by removing this rook and replacing it by two rooks in the squares it attacks in these words, as well as adding a rook immediately to the right of $B_{n - 1}$. We now show that there is always a second rook placement on $\Grid{w}$ by considering several cases, each of which involves identifying a small number of rooks that can be rearranged in the rook placement on $\Grid{w}$ inherited from the unique placement on $\Grid{w'}$ in order to give a second rook placement on $\Grid{w}$.

If $B_j$ does not occur in the row directly to the north of $B_n$, then (in the rook placement inherited from $w'$) there is a rook $R'$ directly below $B_j$ which, in $\Grid{w}$, attacks the words to the north and to the west of $B_n$, and so replacing this rook by two rooks yields a second rook placement on $\Grid{w}$. Hence, we may assume that $B_j$ occurs in the row directly to the north of $B_n$.

If both the NE and SW quadrants are empty (other than the overlapping square $B_c$), then $w'$ is an increasing sequence (that is, just the identity permutation). Then $w$ can be written as the union of an increasing and decreasing sequence which overlap by choosing the $n, n-1$ to be the decreasing sequence and $1,2,\dots, n-2$ to be the increasing sequence. This contradicts our assumption on $w$, and so there must be a black square in either the NE or SW quadrant. 

\textbf{Case 2a:} Suppose there is a black square, say $B_k$, in the NE quadrant. In this case, in the unique rook placement on $\Grid{w'}$, there is a rook $S$ directly to the east of $B_k$. Necessarily, in the inherited placement on $\Grid{w}$, $S$ will attack the words directly to the north and directly to the west of $B_n$, so replacing $S$ by these two rooks yields a second rook placement on $\Grid{w}$.

\textbf{Case 2b:} Suppose there are no black squares in the NE quadrant. Then there is a black square, say $B_k$, in the SW quadrant. Let $B_{\ell}$ denote the northernmost black square which is south of $B_n$ in the SE quadrant. We claim that some square in the SW quadrant must occur to the north of $B_{\ell}$. If not, we can write $w$ as a union of an overlapping increasing and decreasing sequence by choosing the decreasing sequence to be $n$, $\ell$, and the black squares in the SW quadrant, contradicting our hypothesis on $w$. Hence, we can take $B_k$ to be north of $B_{\ell}$; in fact, we can take $B_k$ to occur in the row directly to the south of $B_n$. 

Recall that $B_j$ is a black square in the SW quadrant which occurs in the row directly to the north of $B_n$, and that we have already identified one rook placement on $\Grid{w}$ which arises from replacing the rook $R$ that is immediately to the east of $B_j$ in the rook placement inherited from $\Grid{w'}$ with rooks $R_1$ and $R_2$ in the squares that it attacks in the word to the west and to the north of $B_n$, respectively.  In this rook placement, there is also a rook $R'$ two squares to the south of $B_j$ (inherited from the rook directly below $B_j$ in the placement on $\Grid{w'}$). Then $R_1$ and $R'$ mutually attack two squares, and so we obtain a second rook placement on $\Grid{w}$ by Lemma~\ref{lem.mutualrooks}.

The list of cases above is exhaustive, and in all cases we find that if $w$ is not the union of an increasing and decreasing subsequence that overlap, then $\Grid{w}$ admits at least two rook placements.  This completes the proof.
\end{proof}

The permutations that appear in Theorem~\ref{thm:one placement} are an interesting class of permutations in their own right.  As observed the proof of \cite[Lemma~9]{atk98}, they are precisely the permutations whose image under the Robinson--Schensted correspondence is of hook shape, and consequently there are $\binom{2(n - 1)}{n - 1}$ of these in $S_n$.  (For more background on this enumeration, see \cite{strehl98}.)  These permutations belong to the \emph{skew-merged permutations}, namely, those that can be partitioned into an increasing and decreasing subsequence \cite{atk98, sta94}, or equivalently those that avoid the two patterns $3~4~1~2$ and $2~1~4~3$ \cite[Theorem 2.9]{sta94}. It turns out that the other skew-merged permutations are also interesting from our point of view, as we describe in the next proposition and conjecture.

\begin{proposition}
\label{prop:skewmerged}
If a permutation $w \in S_n$ is skew-merged, then $\Grid{w}$ admits exactly one or two rook placements.
\end{proposition}
\begin{proof}
By Theorem~\ref{thm:one placement}, $\Grid{w}$ admits exactly one rook placement if and only if $w$ can be written as the union of an increasing and decreasing subsequence which overlap. Hence, it suffices to show that if $w$ can be written as the union of an increasing and decreasing subsequence which do not overlap, then $\Grid{w}$ admits exactly two rook placements.

By \cite[Lemma 3]{atk98}, if $w$ is a skew-merged permutation, then in $\Grid{w}$, every black square can be uniquely assigned to one of five regions in the grid: NW, NE, SW, SE, or center. Given two black squares in the NW region, the further east square is also further south. Similarly, in the NE region, further east squares are further north; in the SW region, further east squares are further north; and in the SE region, further east squares are further south.  The NW and SE regions contain the black squares which must be in the increasing subsequence in any partition of $w$ into an increasing and decreasing subsequence. Meanwhile the SW and NE regions contain the black squares which must be in the decreasing subsequence. The center region contains those squares that may be chosen to be in either the increasing or decreasing subsequence in a partition of $w$.

Suppose that $w$ can be written as the union of an increasing and decreasing subsequence which do not overlap. This means that in $\Grid{w}$ there are no black squares in the center region, and so every black square may be unambiguously assigned to the NW, NE, SW, or SE region. It follows from \cite[Lemma 11]{atk98} that all four regions are nonempty and if we consider the furthest southeast black square in the NW, the furthest southwest in the NE, the furthest northeast in the SW, and the furthest northwest in the SE, then these four black squares must form the permutation $2~4~1~3$ (in the order NW, NE, SW, SE) or the permutation $3~1~4~2$ (in the order NE, NW, SE, SW); this is illustrated in Figure~\ref{fig:skew-merged center}.  Refer to these four squares as $B_1,B_2,B_3,B_4$, where $B_i$ corresponds to $i$ in the permutation pattern.

\begin{figure}[ht]
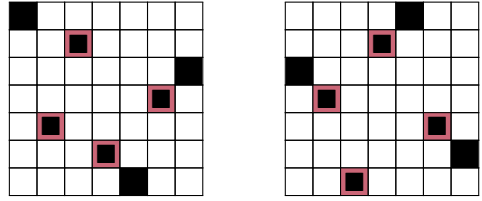

\begin{center}
\begin{ytableau}
*(black)&&&&&& \\
&&*(color1)\blacksquarenew&&&&\\
&&&&&&*(black)\\
&&&&&*(color1)\blacksquarenew&\\
&*(color1)\blacksquarenew&&&&&\\
&&&*(color1)\blacksquarenew&&&\\
&&&&*(black)&&
\end{ytableau}
\hspace{24pt}
\begin{ytableau}
*(white)&&&&*(black)&& \\
&&&*(color1)\blacksquarenew&&&\\
*(black)&&&&&&\\
&*(color1)\blacksquarenew&&&&&\\
&&&&&*(color1)\blacksquarenew&\\
&&&&&&*(black)\\
&&*(color1)\blacksquarenew&&&&
\end{ytableau}
\end{center}
\caption{The skew-merged permutations $1~3~7~6~2~4~5$ and $5~4~1~2~6~7~3$ can be written in a unique way as the union of an increasing and decreasing subsequence, and these do not overlap.  The extreme squares in each of the four quadrants are highlighted in red.}
\label{fig:skew-merged center}
\end{figure}

The permutations that occur in the $3~1~4~2$ case are precisely the inverses of the permutations that occur in the $2~4~1~3$ case.  Since $\Grid{w}$ and $\Grid{w^{-1}}$ obviously admit the same rook placements (with the matrix transpose providing a bijection), it suffices to consider the  $2~4~1~3$ case.  In this case, by \cite[Lemma 11]{atk98}, there is a column $j$ such that that $B_2$ is in column $j$ and $B_3$ is in column $j + 1$. By a rotated version of the same lemma, there is also a row $i$ such that $B_4$ is in row $i$ and $B_1$ is in row $i + 1$. Thus, the vertical line separating columns $j$ and $j + 1$ and the horizontal line separating rows $i$ and $i + 1$ divide the square into precisely the four quadrants already mentioned.

By the same proof as in Theorem~\ref{thm:one placement}, in any rook placement on $\Grid{w}$, there must be rooks placed directly to the north of any black square in the NW or NE quadrants, directly to the south of any black square in the SW or SE quadrants, directly to the west of any black square in the NW or SW quadrants, and directly to the east of any black square in the NE or SE quadrants, as illustrated in Figure~\ref{fig:two placements}.
\begin{figure}[ht]
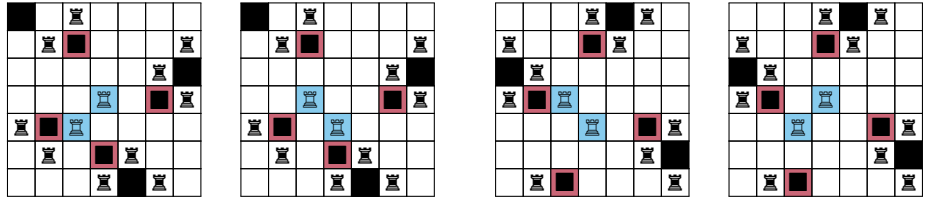

\begin{center}
\begin{ytableau}
*(black)&&\Rook&&&& \\
&\Rook&*(color1)\blacksquarenew&&&&\Rook\\
&&&&&\Rook&*(black)\\
&&&*(color3)\WRook&&*(color1)\blacksquarenew&\Rook\\
\Rook&*(color1)\blacksquarenew&*(color3)\WRook&&&&\\
&\Rook&&*(color1)\blacksquarenew&\Rook&&\\
&&&\Rook&*(black)&\Rook&
\end{ytableau}
\hspace{8pt}
\begin{ytableau}
*(black)&&\Rook&&&& \\
&\Rook&*(color1)\blacksquarenew&&&&\Rook\\
&&&&&\Rook&*(black)\\
&&*(color3)\WRook&&&*(color1)\blacksquarenew&\Rook\\
\Rook&*(color1)\blacksquarenew&&*(color3)\WRook&&&\\
&\Rook&&*(color1)\blacksquarenew&\Rook&&\\
&&&\Rook&*(black)&\Rook&
\end{ytableau}
\hspace{16pt}
\begin{ytableau}
*(white)&&&\Rook&*(black)&\Rook& \\
\Rook&&&*(color1)\blacksquarenew&\Rook&&\\
*(black)&\Rook&&&&&\\
\Rook&*(color1)\blacksquarenew&*(color3)\WRook&&&&\\
&&&*(color3)\WRook&&*(color1)\blacksquarenew&\Rook\\
&&&&&\Rook&*(black)\\
&\Rook&*(color1)\blacksquarenew&&&&\Rook
\end{ytableau}
\hspace{8pt}
\begin{ytableau}
*(white)&&&\Rook&*(black)&\Rook& \\
\Rook&&&*(color1)\blacksquarenew&\Rook&&\\
*(black)&\Rook&&&&&\\
\Rook&*(color1)\blacksquarenew&&*(color3)\WRook&&&\\
&&*(color3)\WRook&&&*(color1)\blacksquarenew&\Rook\\
&&&&&\Rook&*(black)\\
&\Rook&*(color1)\blacksquarenew&&&&\Rook
\end{ytableau}
\end{center}
\caption{For the permutations $w$ of Figure~\ref{fig:skew-merged center}, in every rook placement on $\Grid{w}$, the black rooks are forced, as in the proof of Theorem~\ref{thm:one placement}.  The white rooks on blue background are the two possible ways to complete each placement.}
\label{fig:two placements}
\end{figure}
This accounts for $2n - 4$ total rooks (two adjacent to each black square, with a ``missing'' rook along each edge).  Thus, every rook placement on the grid requires two additional rooks.  The words which do not contain any rooks are precisely the down to the south of $B_2$, the down to the north of $B_3$, the across to the east of $B_1$ and the across to the west of $B_4$. There are exactly two ways to complete this rook placement: with rooks in $(i,j)$ and $(i + 1, j +1)$ or with rooks in $(i, j+1)$ and $(i+1,j)$ (as illustrated in Figure~\ref{fig:two placements}). This completes the proof. 
\end{proof}

We conjecture that the converse to Proposition~\ref{prop:skewmerged} holds.

\begin{conj}
\label{conj:skewmerged}
    A permutation $w \in S_n$ is skew-merged if and only if $\Grid{w}$ admits exactly one or two rook placements.
\end{conj}

A consequence of Conjecture~\ref{conj:skewmerged} would be that the number of $n \times n$ permutation grids which admit exactly two rook placements is 
    \[
    \binom{2n}{n} - \binom{2n-2}{n-1} - \sum_{m=0}^{n-1} 2^{n-m-1}\binom{2m}{m}.
    \]

The same general proof strategy used in the proof of Theorem~\ref{thm:one placement} should also apply to Conjecture~\ref{conj:skewmerged}. One can again argue by induction on $n$, but the inductive step is more involved because the corresponding $(n-1)\times(n-1)$ instance may admit either one or two rook placements. As a result, the proof would require a more intricate case analysis. It would be interesting to find a proof of both Theorem~\ref{thm:one placement} and Conjecture~\ref{conj:skewmerged} that avoids such complicated case-by-case analysis.

Theorem~\ref{thm:one placement} and Proposition~\ref{prop:skewmerged} raise the question of which values $r$ can be achieved as the rook placement count of a permutation grid.   
Exhaustive computer computation for $n \leq 8$ shows that for all $r \leq 108$ other than the values $r = 4, 12,$ and all numbers congruent to 3 modulo 4, there exists a permutation $w$ such that $|\RP(\Grid{w})| = r$. We therefore conjecture the following.

\begin{conj}
\label{conj:rookcounts}
A positive integer $r$ occurs as the rook placement count of a permutation grid if and only if $r \neq 4$, $r \neq 12$, and $r \not \equiv 3 \pmod{4}$.
\end{conj}

It follows from Theorem~\ref{thm:one placement} that the set $P$ of permutations which admit exactly one rook placement is not a pattern class. For example, $2~5~3~1~4$ is in $P$ while $2~4~1~3$ is not in $P$.  Similarly, the permutations that admit at most $5$ rook placements are not a pattern class, because $2~6~3~1~5~4$ admits $5$ placements but it contains $2~5~1~4~3$ as a pattern, which admits $6$.  However, it appears (from the same exhaustive computations) that the permutations that admit at most $6$ rook placements may form a pattern class.

\begin{question}
\label{q:6exactlykgeneral}
Fix a positive integer $r$. Can the permutation grids which admit at most (or exactly) $r$ rook placements be characterized? When are they related to permutation pattern classes? 
\end{question}

As mentioned above, the permutations $w$ such that $\Grid{w}$ admits exactly one rook placement can be characterized by their shape under the Robinson--Schensted correspondence.  Casual observation suggests that permutations whose grids admit small numbers of rook placements may always have RS-shapes that are ``close to hooks'' in some sense.  This case leads us to pose the following question.
\begin{question}
    \label{q:RS}
    In general, is the number of rook placements on $\Grid{w}$ related to the Robinson--Schensted shape of $w$?
\end{question}

\subsection{Permutation grids admitting large numbers of placements}
\label{subsec:largenumbers}

The preceding sections we addressed permutation grids that admit small numbers of rook placements.  It is also natural to consider the opposite extreme.
\begin{question}
\label{q:8maxperm}
    Which $n \times n$ permutation grids admit the most rook placements? How many rook placements do they admit? How does this quantity behave asymptotically as $n \to \infty$?
\end{question}
Let $r(n) = \max_{w \in S_n} |\RP(\Grid{w})|$ denote the maximum number of rook placements over all $n \times n$ permutation grids. For $1 \leq n \leq 8$, exhaustive computer computation shows that the sequence $r(n)$ begins $1, 1, 1, 5, 17, 73, 469, 3493$. This sequence does not appear in the On-Line Encyclopedia of Integer Sequences. For these first eight cases, the permutation grid admitting the most rook placements is a block-diagonal permutation whose first and last blocks are anti-diagonal of the same size and whose central block is an identity matrix. In other words, they are layered permutations whose block-sizes are indexed by compositions $(m, 1, 1, \ldots, 1, m)$ for some $m$. For a layered permutation, we call the associated composition its \emph{shape}. 

\begin{conj}
\label{conj:maxlayered}
For each $n$, the $n \times n$ permutation grid admitting the most rook placements is given by $\Grid{w}$ where $w \in S_n$ is a layered permutation of shape $(m,1,1,\dots, 1,m)$ for some $m$.
\end{conj}

For layered permutations of shape $(m_1, 1, 1, \dots, 1, m_2)$, we can compute the number of rook placements on the permutation grid exactly. Recall that the \emph{Stirling number of the second kind} $S(a, b)$ is the number of ways to partition a set of $a$ objects into $b$ non-empty sets.

\begin{prop}
\label{prop:countlayered}
    Suppose $m_1, m_2 > 0$ and let $m = \min(m_1, m_2)$. For the layered permutation $w \in S_n$ of shape $(m_1, 1,1,\dots,1, m_2)$, the number of rook placements on $\Grid{w}$ is given by 
    \[
    \sum_{i = 0}^{m-1} S(m_1, m - i)S(m_2, m - i)\left((m-i)^{n - m_1 - m_2} \cdot (m-i)!\right)^2.
    \]
    In particular, if $m > 1$ then for the layered permutation in $S_n$ of shape $(m,1,1,\dots,1,m)$, the number of rook placements on $\Grid{w}$ is given by 
    \[
    \sum_{k = 1}^m \left[S(m, k) \cdot k^{n-2m} \cdot k!\right]^2.
    \]
\end{prop}
\begin{proof}
Suppose that $w \in S_n$ is the layered permutation of shape $(m_1, 1, 1, \dots, 1, m_2)$. Let $k = n - m_1 - m_2$ denote the number of $1$'s. For example, if $w = (6,1,1,1,1,1,3) \in S_{14}$, then $m_1 = 6$, $m_2 = 3$, $m = \min(m_1, m_2) = 3$, and $k = 5$, and we can picture $\Grid{w}$ as 
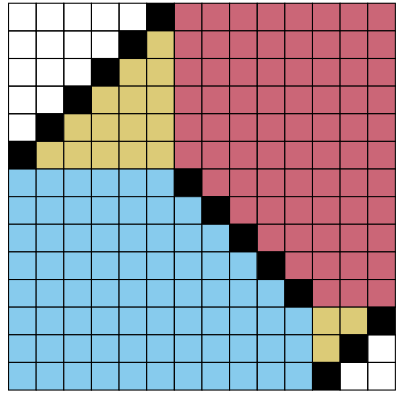
\[
\begin{ytableau}
*(white)&&&&&*(black)& *(color1)&*(color1)&*(color1)&*(color1)&*(color1)&*(color1)&*(color1)&*(color1) \\
&&&&*(black)&*(color2)&*(color1)&*(color1)&*(color1)&*(color1)&*(color1)&*(color1)&*(color1)&*(color1) \\
&&&*(black)&*(color2)&*(color2)&*(color1)&*(color1)&*(color1)&*(color1)&*(color1)&*(color1)&*(color1)&*(color1) \\
&&*(black)&*(color2)&*(color2)&*(color2)&*(color1)&*(color1)&*(color1)&*(color1)&*(color1)&*(color1)&*(color1)&*(color1)\\
&*(black)&*(color2)&*(color2)&*(color2)&*(color2)&*(color1)&*(color1)&*(color1)&*(color1)&*(color1)&*(color1)&*(color1)&*(color1)\\
*(black)&*(color2)&*(color2)&*(color2)&*(color2)&*(color2)&*(color1)&*(color1)&*(color1)&*(color1)&*(color1)&*(color1)&*(color1)&*(color1) \\
*(color3)&*(color3)&*(color3)&*(color3)&*(color3)&*(color3)&*(black)&*(color1)&*(color1)&*(color1)&*(color1)&*(color1)&*(color1)&*(color1) \\
*(color3)&*(color3)&*(color3)&*(color3)&*(color3)&*(color3)&*(color3)&*(black)&*(color1)&*(color1)&*(color1)&*(color1)&*(color1)&*(color1) \\
*(color3)&*(color3)&*(color3)&*(color3)&*(color3)&*(color3)&*(color3)&*(color3)&*(black)&*(color1)&*(color1)&*(color1)&*(color1)&*(color1) \\
*(color3)&*(color3)&*(color3)&*(color3)&*(color3)&*(color3)&*(color3)&*(color3)&*(color3)&*(black)&*(color1)&*(color1)&*(color1)&*(color1) \\
*(color3)&*(color3)&*(color3)&*(color3)&*(color3)&*(color3)&*(color3)&*(color3)&*(color3)&*(color3)&*(black)&*(color1)&*(color1)& *(color1)\\
*(color3)&*(color3)&*(color3)&*(color3)&*(color3)&*(color3)&*(color3)&*(color3)&*(color3)&*(color3)&*(color3)&*(color2)& *(color2) & *(black)\\
*(color3)&*(color3)&*(color3)&*(color3)&*(color3)&*(color3)&*(color3)&*(color3)&*(color3)&*(color3)&*(color3)& *(color2) & *(black)&\\
*(color3)&*(color3)&*(color3)&*(color3)&*(color3)&*(color3)&*(color3)&*(color3)&*(color3)&*(color3)& *(color3) & *(black) &&
\end{ytableau}
\] 
We divide $\Grid{w}$ into six regions, which we have colored above with two white regions, a red region, a blue region, and two yellow regions, as follows: in the NW corner, the $m_1 \times m_1$ block of squares is divided into a white triangular region and a yellow triangular region. Similarly, in the SE corner, the $m_2 \times m_2$ block of squares is divided into a white and yellow region. In any rook placement on $\Grid{w}$, there must be $m_1 - 1$ rooks in the NW white region, with the rooks placed immediately above the black squares. Similarly, there are $m_2 - 1$ rooks in the SE white region. The remaining $2n - 2 - (m_1 - 1) - (m_2 - 1) = 2n - m_1 - m_2$ rooks are in the red, blue, and yellow regions. 

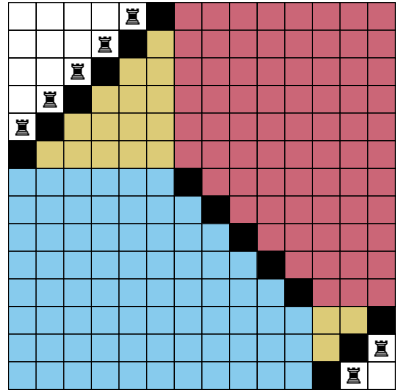
\[
\begin{ytableau}
*(white)&&&&\Rook&*(black)& *(color1)&*(color1)&*(color1)&*(color1)&*(color1)&*(color1)&*(color1)&*(color1) \\
&&&\Rook&*(black)&*(color2)&*(color1)&*(color1)&*(color1)&*(color1)&*(color1)&*(color1)&*(color1)&*(color1) \\
&&\Rook&*(black)&*(color2)&*(color2)&*(color1)&*(color1)&*(color1)&*(color1)&*(color1)&*(color1)&*(color1)&*(color1) \\
&\Rook&*(black)&*(color2)&*(color2)&*(color2)&*(color1)&*(color1)&*(color1)&*(color1)&*(color1)&*(color1)&*(color1)&*(color1)\\
\Rook&*(black)&*(color2)&*(color2)&*(color2)&*(color2)&*(color1)&*(color1)&*(color1)&*(color1)&*(color1)&*(color1)&*(color1)&*(color1)\\
*(black)&*(color2)&*(color2)&*(color2)&*(color2)&*(color2)&*(color1)&*(color1)&*(color1)&*(color1)&*(color1)&*(color1)&*(color1)&*(color1) \\
*(color3)&*(color3)&*(color3)&*(color3)&*(color3)&*(color3)&*(black)&*(color1)&*(color1)&*(color1)&*(color1)&*(color1)&*(color1)&*(color1) \\
*(color3)&*(color3)&*(color3)&*(color3)&*(color3)&*(color3)&*(color3)&*(black)&*(color1)&*(color1)&*(color1)&*(color1)&*(color1)&*(color1) \\
*(color3)&*(color3)&*(color3)&*(color3)&*(color3)&*(color3)&*(color3)&*(color3)&*(black)&*(color1)&*(color1)&*(color1)&*(color1)&*(color1) \\
*(color3)&*(color3)&*(color3)&*(color3)&*(color3)&*(color3)&*(color3)&*(color3)&*(color3)&*(black)&*(color1)&*(color1)&*(color1)&*(color1) \\
*(color3)&*(color3)&*(color3)&*(color3)&*(color3)&*(color3)&*(color3)&*(color3)&*(color3)&*(color3)&*(black)&*(color1)&*(color1)& *(color1)\\
*(color3)&*(color3)&*(color3)&*(color3)&*(color3)&*(color3)&*(color3)&*(color3)&*(color3)&*(color3)&*(color3)&*(color2)& *(color2) & *(black)\\
*(color3)&*(color3)&*(color3)&*(color3)&*(color3)&*(color3)&*(color3)&*(color3)&*(color3)&*(color3)&*(color3)& *(color2) & *(black)&\Rook\\
*(color3)&*(color3)&*(color3)&*(color3)&*(color3)&*(color3)&*(color3)&*(color3)&*(color3)&*(color3)& *(color3) & *(black) &\Rook&
\end{ytableau}
\]

The red region forms a rectangle in the NE corner with a triangle removed from its SW corner. It contains the squares in the northernmost  $n - m_2$ rows and  the easternmost $n - m_1$ columns, other than the squares south of the $k$ black squares in the middle of $\Grid{w}$. Therefore, there are $n- m_2$ across words which contain any red squares and $n- m_1$ down words which contain any red squares. Without loss of generality, assume that $m_2 \leq m_1$, so that the SE yellow region is no larger than the NW yellow region.

Consider the words that contain red squares.  A rook in each such word must belong to either the red or one of the yellow regions.  Each rook in the red region attacks a down word and an across word from this collection, while each rook in the SE yellow region attacks only a down word that intersects the red region and each rook the NW yellow region attacks only an across word that intersects the red region.  Therefore, in any rook placement on $\Grid{w}$, the NW yellow region must contain $m_1 - m_2$ more rooks than the SE yellow region. 

We now count the number of rook placements on $\Grid{w}$ by considering how many rooks are in the SE yellow region. Consider a rook placement in which the SE yellow region contains $i$ rooks, whence the NW yellow region contains $i + (m_1 - m_2)$ rooks. These rooks must occur in distinct across words, so $0 \leq i \leq m_2 - 1$. There are exactly $S(m_2, m_2 - i) = S(m_2, m - i)$ ways to place $i$ nonattacking rooks in the SE yellow region \cite[\S5]{kr}. Similarly, there are $S(m_1, m_1 - (i + m_1 - m_2) = S(m_1, m - i)$ ways to place the rooks in the NW yellow region.

For each placement of $i$ rooks in the SE yellow region and $i + m_1 - m_2$ rooks in the NW yellow region, the squares in the red region which are not attacked by any of these rooks form $n - m_1 - i$ rows, of lengths $m_2 - i$, $m_2 - i + 1$, \dots, $n - m_1 - i$, $n-m_1-i$, $n-m_1-i, \dots$. We illustrate this in our running example where we have placed $i = 1$ rook in the SE yellow region (and hence $i + m_1 - m_2 = 4$ rooks in the NW yellow region) and faded the color of every square which is attacked by a rook. The unattacked portion of the red region consists of exactly $n - m_1 - i = 7$ rows, of lengths $2,3, 4,5, 6, 7, 7$: 
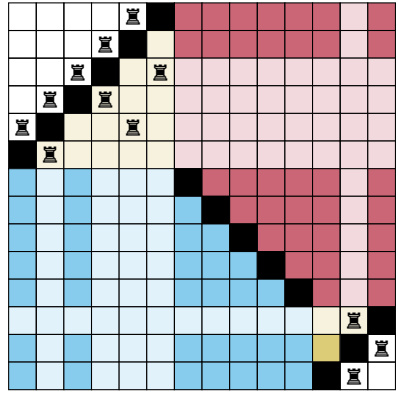
\[
\begin{ytableau}
*(white)&&&&\Rook&*(black)& *(color1)&*(color1)&*(color1)&*(color1)&*(color1)&*(color1)&*(lc1)&*(color1) \\
&&&\Rook&*(black)&*(lc2)&*(color1)&*(color1)&*(color1)&*(color1)&*(color1)&*(color1)&*(lc1)&*(color1) \\
&&\Rook&*(black)&*(lc2)&*(lc2)\Rook&*(lc1)&*(lc1)&*(lc1)&*(lc1)&*(lc1)&*(lc1)&*(lc1)&*(lc1) \\
&\Rook&*(black)&*(lc2)\Rook&*(lc2)&*(lc2)&*(lc1)&*(lc1)&*(lc1)&*(lc1)&*(lc1)&*(lc1)&*(lc1)&*(lc1)\\
\Rook&*(black)&*(lc2)&*(lc2)&*(lc2)\Rook&*(lc2)&*(lc1)&*(lc1)&*(lc1)&*(lc1)&*(lc1)&*(lc1)&*(lc1)&*(lc1)\\
*(black)&*(lc2)\Rook&*(lc2)&*(lc2)&*(lc2)&*(lc2)&*(lc1)&*(lc1)&*(lc1)&*(lc1)&*(lc1)&*(lc1)&*(lc1)&*(lc1) \\
*(color3)&*(lc3)&*(color3)&*(lc3)&*(lc3)&*(lc3)&*(black)&*(color1)&*(color1)&*(color1)&*(color1)&*(color1)&*(lc1)&*(color1) \\
*(color3)&*(lc3)&*(color3)&*(lc3)&*(lc3)&*(lc3)&*(color3)&*(black)&*(color1)&*(color1)&*(color1)&*(color1)&*(lc1)&*(color1) \\
*(color3)&*(lc3)&*(color3)&*(lc3)&*(lc3)&*(lc3)&*(color3)&*(color3)&*(black)&*(color1)&*(color1)&*(color1)&*(lc1)&*(color1) \\
*(color3)&*(lc3)&*(color3)&*(lc3)&*(lc3)&*(lc3)&*(color3)&*(color3)&*(color3)&*(black)&*(color1)&*(color1)&*(lc1)&*(color1) \\
*(color3)&*(lc3)&*(color3)&*(lc3)&*(lc3)&*(lc3)&*(color3)&*(color3)&*(color3)&*(color3)&*(black)&*(color1)&*(lc1)& *(color1)\\
*(lc3)&*(lc3)&*(lc3)&*(lc3)&*(lc3)&*(lc3)&*(lc3)&*(lc3)&*(lc3)&*(lc3)&*(lc3)& *(lc2)& *(lc2)\Rook & *(black)\\
*(color3)&*(lc3)&*(color3)&*(lc3)&*(lc3)&*(lc3)&*(color3)&*(color3)&*(color3)&*(color3)&*(color3)& *(color2) & *(black)&\Rook\\
*(color3)&*(lc3)&*(color3)&*(lc3)&*(lc3)&*(lc3)&*(color3)&*(color3)&*(color3)&*(color3)& *(color3) & *(black) &\Rook&
\end{ytableau}
\]
Since no rooks in the blue region attack any red squares (and vice versa), each completion of this partial rook placement consists of two independent pieces, a set of red rooks that attack the remaining red squares and a set of blue rooks that attack the remaining blue squares.  Moreover, by discarding the irrelevant already-attacked rows and columns, we see that the number of rook placements in both the red and blue regions is the same as the number of rook placements on an $(n - m_1 - i) \times (n - m_1 - i)$ square with a triangular region removed from a corner so that the southernmost row contains $m - i$ squares. We compute this quantity now.

There are $m - i$ ways to place a rook in the southernmost row, after which there are $m - i$ ways to place a rook in the next row (regardless of where the first rook was placed), and this continues for the southernmost $n - m_1 - m_2$ rows. After placing these $n - m_1 - m_2$ rooks, the remaining red squares which are not attacked form an $(m - i) \times (m-i)$ square, which admits $(m - i)!$ rooks placements. Altogether, therefore have $(m - i)^{n - m_1 - m_2} \cdot (m - i)!$ rook placements in each of the red and blue regions.  

Combining this with the previous calculations, there are $S(m_1, m- i)S(m_2, m-i)[(m - i)^{n - m_1 - m_2} \cdot (m - i)!]^2$ rook placements with $i$ rooks in the SE yellow regions. By letting $i$ run between $0$ and $m - 1$ and summing, we obtain the desired formula.
\end{proof}

In light of Conjecture~\ref{conj:maxlayered}, it is natural to ask which values of $m_1, m_2$ (in terms of $n$) maximize the number of rook placements given by Proposition~\ref{prop:countlayered}.  The following proposition shows that the maximum is obtained when $m_1$ is equal to $m_2$, and gives a small interval in which this common value must lie.

\begin{proposition}
\label{prop:layeredmax}
    For a positive integer $n$, consider the collection of integers
    \[
    \sum_{i = 0}^{m-1} S(m_1, m - i)S(m_2, m - i)\left((m-i)^{n - m_1 - m_2} \cdot (m-i)!\right)^2
    \]
    over all pairs $m_1, m_2$ of integers with $0 \leq m := \min(m_1, m_2)$ and $m_1 + m_2 \leq n$.  Then the largest of these numbers is achieved when
    \[
    m_1 = m_2 \in [m_*, m^*]
    \]
    where $m_*, m^*$ are respectively the (unique) positive solutions of
    $n = \frac{\ln(m_* + 1)}{\ln(1 + 1/m_*)} + 2m_*$ and $n = \frac{\ln\big((3m^* + 1)(m^* - 1))/4\big)}{2 \ln(1 + 1/m^*)} + 2m^* + 1$.
\end{proposition}
\begin{proof}
We begin by proving that $S(n + 1, k) \leq \binom{k + 1}{2} S(n, k)$ for all integers $1 \leq k \leq n$.  For all $n \geq 1$, we have equality for $k = 1$ (both sides are equal to $1$) and $k = n$ (both sides are equal to $\binom{n + 1}{2}$).  Now proceed by induction on the pair $(n, k)$.  For any integers $1 < k < n$, we have
\begin{align*}
S(n + 1, k) & = k S(n, k) + S(n, k - 1) \\
& \leq k \binom{k + 1}{2} S(n - 1, k) +  \binom{k}{2} S(n - 1, k - 1) && \text{(inductive hypothesis twice)} \\
& \leq \binom{k + 1}{2} \big(k S(n - 1, k) + S(n - 1, k - 1)\big) && \text{(triangular numbers are nondecreasing)}\\
& = \binom{k + 1}{2} S(n, k).
\end{align*}
So the claim holds by induction.

It follows immediately that if $m = m_1 < m_2$ then 
\[
S(m_2, m - i) \leq \binom{m - i + 1}{2} S(m_2 - 1, m - i) \leq (m - i)^2 S(m_2 - 1, m - i),
\]
and therefore that
\begin{multline*}
\sum_{i = 0}^{m-1} S(m_1, m - i)S(m_2, m - i)\left((m-i)^{n - m_1 - m_2} \cdot (m-i)!\right)^2
 \leq {}\\
\sum_{i = 0}^{m-1} S(m_1, m - i) \cdot S(m_2 - 1, m - i)\left((m-i)^{n - m_1 - (m_2 - 1)} \cdot (m-i)!\right)^2.
\end{multline*}
By repeating this $m_2 - m_1$ times, we conclude that the maximum is achieved for a pair with $m = m_1 = m_2$, in which case the sum in question can be rewritten as
 \begin{equation}
 \label{eq:sum in symmetic case}
 \sum_{k = 1}^m \left(S(m, k) \cdot k^{n-2m} \cdot k!\right)^2.
 \end{equation}
We now consider which value of $m$ maximizes this expression.

The function $\frac{\ln(m + 1)}{\ln(1 + 1/m)} + 2m$ is obviously continuous and increasing for $m \geq 1$, so for $n \geq 3$ there is a unique real number $m_*$ such that $n = \frac{\ln(m_* + 1)}{\ln(1 + 1/m_*)} + 2m_*$.  Then for $1 \leq k \leq m \leq m_*$ we have
\begin{align*}
\frac{S(m + 1, k + 1) (k + 1)^{n - 2(m + 1)} (k + 1)!}{S(m, k) k^{n - 2m} k! } & \geq \frac{(k + 1)^{n - 2m - 1}}{k^{n - 2m}} \\
& = \frac{1}{k + 1} \left(1 + \frac{1}{k}\right)^{n - 2m} \\
& \geq \frac{1}{m_* + 1 } \left(\frac{m_* + 1}{m_*}\right)^{n - 2m_*}  = 1.
\end{align*}
Clearing the denominator, squaring both sides, and summing over $k$ gives
\begin{align*}
\sum_{k = 1}^m \left(S(m, k) \cdot k^{n-2m} \cdot k!\right)^2 & \leq \sum_{k = 1}^{m} \left(S(m + 1, k + 1) \cdot (k + 1)^{n-2(m + 1)} \cdot (k + 1)!\right)^2 \\
& < \sum_{k = 1}^{m + 1} \left(S(m + 1, k) \cdot k^{n-2(m + 1)} \cdot k!\right)^2.
\end{align*}
Hence the maximum value of \eqref{eq:sum in symmetic case} is obtained for $m > m_*$.

For any $n$ and for $1 \leq k \leq m - 1$ we have
\begin{equation}
\label{eq:decreasing terms}    
S(m, k) k^{n - 2m} k! \geq \binom{k + 1}{2} S(m, k) k^{n - 2m - 2} k! \geq S(m + 1, k)k^{n - 2m - 2} k!.
\end{equation}
The function $\frac{\ln\big((3m + 1)(m - 1))/4\big)}{2 \ln(1 + 1/m)} + 2m + 1$ is obviously continuous and increasing for $m > 1$, so there is a unique real number $m^*$ such that $n = \frac{\ln\big((3m^* + 1)(m^* - 1))/4\big)}{2 \ln(1 + 1/m^*)} + 2m^* + 1$.  Then for $m \geq m^*$ we have
\[
2(n - 2m - 1) \ln(1 + 1/m) \leq \ln\frac{3m^2 - 2m - 1}{4},
\]
and hence
\[
\left(\frac{m + 1}{m}\right)^{2(n - 2m - 1)} \leq \frac{3m^2 - 2m - 1}{4} = m^2 - \frac{(m + 1)^2}{4}.
\]
Multiplying both sides through by $m^{2(n - 2m - 1)}$ and rearranging, this is equivalent to
\begin{align*}
 m^{2(n - 2m)} & \geq (m + 1)^2 \cdot (m + 1)^{2(n - 2m - 2)} + \binom{m + 1}{2}^2 \cdot m^{2(n - 2m - 2)} \\
 & = \left(
 (m + 1)^{n - 2(m + 1)} \cdot (m + 1)\right)^2 + \left(S(m + 1, m) m^{n - 2(m + 1)}\right)^2.
\end{align*}
Multiplying through by $(m!)^2$ and summing together with the $m - 1$ terms from \eqref{eq:decreasing terms} gives \[
\sum_{k = 1}^m \left(S(m, k) \cdot k^{n-2m} \cdot k!\right)^2 \geq \sum_{k = 1}^{m + 1} \left(S(m + 1, k) \cdot k^{n-2(m + 1)} \cdot k!\right)^2.
\]
Hence the largest value of \eqref{eq:sum in symmetic case} is obtained for $m \leq m^*$.
\end{proof}

One can show that the length of the interval $[m_*, m^*]$ in Proposition~\ref{prop:layeredmax} grows slowly with $n$, but is unbounded.  Computations suggest that the value $m_*$ is essentially the true optimal value.

\begin{conj}
For a fixed positive integer $n$, 
the maximum value of the expression in \eqref{eq:sum in symmetic case} is obtained when $m \in (m_*, m_* + 2)$.
\end{conj}

Empirically, it seems that for this optimal value of $m$, the largest term of \eqref{eq:sum in symmetic case} dominates the sum.  Thus, have the following simpler lower bound on the asymptotic growth of the $r(n) = \max_{w \in S_n} |\RP(\Grid{w})|$.

\begin{corollary}
For all $n$,
\[
r(n) \geq \left(m^{n - 2m} \cdot m!\right)^2
\]
where $m = \lceil x \rceil$ and $x$ is the unique positive solution of the equation $\ln(1 + x)/\ln(1 + 1/x) + 2x = n$.
\end{corollary}

Expressing the right side directly in terms of $n$ is complicated by the presence of logarithmic terms and the superexponential dependence on $m$.

\section{Questions}
\label{sec:questions}

We end with an assortment of open questions, in addition to those mentioned above.

\subsection{Permutation grids}

In Proposition~\ref{prop:existence}, we showed that for all $w \in S_n$, $\Grid{w}$ admits a rook placement. The proof provides a recursive construction of a rook placement on $\Grid{w}$ from a rook placement on the grid with one row and one column deleted. Is a better description possible?

\begin{question}
\label{q:4RPonperm}
    Is there an explicit (non-recursive) description of one (or more, or all) rook placements on a permutation grid, in terms of the permutation?
\end{question}

\subsection{Which numbers count rook placements?}

Conjecture~\ref{conj:rookcounts} and Question~\ref{q:6exactlykgeneral} raise the question of which integers $r$ can be the number of rook placements on a permutation grid (and what this says about the structure of the permutation). Questions about rook placement counts are interesting even outside of the permutation grid setting.

\begin{question}
\label{q:7whichk}
    For which positive integers $r$ does there exist a crossword grid which admits exactly $r$ rook placements? Can the crossword grids which admit exactly $r$ rook placements be characterized?
\end{question}

Suppose there exist crossword grids $G_1$ and $G_2$ of size $n_1 \times n_1$ and $n_2 \times n_2$ admitting $r_1$ and $r_2$ rook placements, respectively. Then we can construct an $(n_1 + n_2) \times (n_1 + n_2)$ grid by placing $G_1$ in the NW corner and $G_2$ in the SE corner with the remaining squares being black. This grid clearly admits $r_1r_2$ rook placements, so it suffices to answer Question~\ref{q:7whichk} for prime $r$.

\subsection{Average behavior}

Above we have mostly focused on extreme numbers of rook placements, but one might also consider average behavior. For example, of the $16$ possible $2 \times 2$ grids (with no symmetry or connectedness condition), $12$ have the same number of across and down words; these all admit rook placements, and the expected number of rook placements among them is $13/12$ (or $13/16$ over all $2 \times 2$ grids). Of the $3 \times 3$ grids, $234$ of $512$ have the same number of across and down words; of those, $209/234$ admit rook placements, and the expected number of rook placements among them is $266/234$ (or $266/512$ over all $3 \times 3$ grids). 

\begin{question}
Fix a positive integer $n$. What fraction of $n \times n$ grids admit a rook placement? How does this fraction behave asymptotically as $n \to \infty$? What is the expected number of rook placements on a random $n \times n$ grid?  
\end{question}

One could also ask the same questions conditional on the number (density) of black squares in the grid. Of course, the expected number of rook placements on a grid with $k=0$ black squares is $n!$. Since we computed exactly how many rook placements a grid with $1$ black square admits (just before Question~\ref{q:2mostperk}), we can also see that as $n \to \infty$, the expected number of rook placements on a grid with $k = 1$ black square is $n^4 (n - 3)! \cdot \int_0^1\int_0^1 a(1 - a)b(1 - b) \, da \, db = \frac{n^4}{36} (n - 3)!$.
For $k$ finite as $n \to \infty$, it should follow that the expected number of rook placements is  asymptotic to $\frac{n^k}{36^k} n!$, since in this case the probability that two black squares end up in the same row or column tends to $0$.  

It would be interesting to study the average number of rook placements for $k$ black squares as $k$ varies. When $k = o(\sqrt{n})$ the expected number of pairs of black squares in the same row or column is $\leq \binom{k}{2} \cdot \frac{2}{n} < \frac{k^2}{n} = o(1)$, so in this regime one might still hope for the same asymptotic average $\frac{n^k}{36^k} n!$ as $n \to \infty$. 

When $k = o(n)$, the probability that two black squares are adjacent or that there are any black squares on the boundary is at most $\binom{k}{2} \cdot \frac{4}{n^2} + k \cdot \frac{4n}{n^2} < \frac{2k^2}{n^2} + \frac{4k}{n} \to 0$, so the probability that there are the same number of down and across words tends to 1 as $n \to \infty$. What is the average number of rook placements in this regime? What can be said when $n \ll k \ll n^2$ or when $k = \Theta(n^2)$?

\subsection{ASMs}

In light of Proposition~\ref{prop:ASMs}, the remaining results of Section~\ref{sec:permgrids} can be viewed as counting or bounding the number of ASMs whose $-1$s occupy a prescribed set of positions.  Thus, all the questions we've posed can be naturally carried over to this context.  For example, which sets of positions for the $-1$s admit at least one ASM (i.e., which sparse crossword grids admit a rook placement)?  Which admit a unique ASM?  Which admit the largest number of ASMs, and what is this maximum?

\subsection{Fractional rook placements?}

The process of $k$-inflation (see Lemma~\ref{lem:inflation}) suggests a theory of \emph{fractional rook placements}, in which each white square of the grid is filled with a real number in $[0, 1]$, and the sum of the numbers in each word must be equal to $1$. (For a reference in fractional matching theory on graphs, see \cite{su11}.)

\begin{question}
\label{q:fracrook}
Can anything interesting be said about fractional rook placements on crossword grids? 
\end{question}

\subsection{Non-square grids}

With the exception of Proposition~\ref{prop:upper bound}, we have chosen to focus on square grids in this paper; however, many of the same questions can be asked on grids of general size.

\begin{question}
\label{q:9nonsquare}
What can be said about rook placements on non-square grids?
\end{question}

\subsection{A complete rook theory}

In classical rook theory, it is of considerable interest to study not just maximum rook placements on a board $B$, but also the \emph{rook polynomial} $R_B(x) = \sum_{k} r_k(B) x^k$ whose coefficient $r_k(B)$ counts the placements of $k$ non-attacking rooks on $B$. The rook polynomial has connections to Stirling numbers, chromatic polynomials, and the representation theory of the symmetric group.

On a grid $G$, it is possible to define a polynomial tracking partial not-necessarily-complete non-attacking rook placements on $G$ by size. This gives a polynomial which is precisely the matching polynomial of $\Graph{G}$.

\begin{question}
    \label{q:10rooktheory}
Using the matching polynomial for $\Graph{G}$, which results in classical rook theory have analogues in this context? Are there $q$-analogues?
\end{question}

\bibliographystyle{amsalpha}
\bibliography{biblio}

\end{document}